\documentclass[12pt,reqno]{amsart}

\usepackage{mathtools}
\usepackage{amsfonts}
\usepackage{amsmath}
\usepackage{amsthm}
\usepackage{ytableau}
\usepackage{amssymb}
\usepackage{multirow}
\usepackage{hyperref}
\usepackage{comment}
\usepackage{natbib}
\usepackage{blindtext}
\usepackage{enumitem}
\usepackage{hyperref}
\usepackage{cleveref}

\newtheorem{theorem}{Theorem}[section]
\newtheorem{lemma}[theorem]{Lemma}
\newtheorem{proposition}[theorem]{Proposition}
\newtheorem{corollary}[theorem]{Corollary}

\theoremstyle{definition}
\newtheorem{definition}[theorem]{Definition}
\newtheorem{question}[theorem]{Question}
\theoremstyle{remark}
\newtheorem{example}[theorem]{Example}
\newtheorem{remark}[theorem]{Remark}
\newtheorem*{remark*}{Remark}
\newtheorem*{question*}{Question}
\newtheorem*{problem*}{The Superrestriction Problem}

\usepackage{tikz}
\usetikzlibrary{matrix,backgrounds}
\usepackage{tikz-cd}
\usetikzlibrary{graphs, shapes}

\newcommand{\T}{T_{m,n}^{d}}

\newcommand{\Hi}{\mathcal{H}(m,n;d)}

\newcommand{\ind}{\mathrm{sInd}_{m,n}^{d}}
\newcommand{\res}{\mathrm{sRes}_{m,n}^{d}}

\DeclareMathOperator{\fr}{{Frob}}
\DeclareMathOperator{\ch}{char}
\DeclareMathOperator{\gl}{\mathrm{GL_{n}}(\mathbb{C})}
\DeclareMathOperator{\glm}{\mathrm{GL_{m}}(\mathbb{C})}
\newcommand{\indc}{\mathrm{Ind}^{d}}

\DeclareMathOperator{\Res}{Res}
\DeclareMathOperator{\Hom}{Hom}
\DeclareMathOperator{\mod1}{-mod}
\DeclareMathOperator{\smod}{-smod}
\DeclareMathOperator{\rep}{-rep}
\DeclareMathOperator{\sign}{sign}
\DeclareMathOperator{\parr}{Par}
\DeclareMathOperator{\stab}{Stab}
\newcommand{\aj}{A_{*j}}
\newcommand{\ai}{A_{i*}}

\title[superrestriction coefficients]{A superplethystic formula for\\superrestriction 
coefficients}
\author{Tirtharaj Basu}
\address{The Institute of Mathematical Sciences\\
A CI of Homi Bhabha National Institute\\
Chennai 600113\\
India}

\keywords{symmetric functions, supersymmetric polynomials, Schur superalgebra, restriction coefficients, unimodality}

\subjclass{17A70, 20C30, 20G43, 05E05, 05E10}

\begin{document}
\begin{abstract}
    The Schur superalgebra $S(m,n;d)$ contains $S_{m}\times S_{n}$ as a multiplicative subgroup. We consider the problem of determining the restriction coefficients when an irreducible supermodule of $S(m,n;d)$ is restricted to $S_{m}\times S_{n}$. We introduce two notions of superplethysm and provide a superplethystic variant of Littlewood's formula for these restriction coefficients. As an application, we provide several unimodality results for bipartite superpartitions.
\end{abstract}
\maketitle
\section{Introduction}
The symmetric group $S_{d}$ acts on the vector superspace $\T=\bigotimes^{d}(\mathbb{C}^{m}\oplus\mathbb{C}^{n})$, and its commutant $\text{End}_{S_{d}}(\T)$, denoted by $S(m,n;d)$, is called the Schur superalgebra. The product $S_{m}\times S_{n}$ sits within the even part of $S(m,n;d)$ as a multiplicative subgroup. 

The irreducible supermodules of $S(m,n;d)$ are indexed by partitions $\lambda\vdash d$ satisfying $\lambda_{m+1}\leq n$. We denote this indexing set by $\Hi$. Let $W_{\lambda}$ denote the irreducible supermodule of $S(m,n;d)$ corresponding to $\lambda\in\Hi$, and $V_{\mu}$ denote the Specht module of $S_n$ corresponding to any partition $\mu\vdash n$.
Let $\res W$ denote the restriction of an $S(m,n;d)$ supermodule to $S_m\times S_n$. In this article, we investigate the following problem.

\begin{problem*}
    Determine the multiplicities $r_{\lambda\mu\nu}$ in the decomposition
    \begin{displaymath}
        \res W_\lambda = \bigoplus_{\mu\vdash m,\;\nu\vdash n} (V_\mu\otimes V_\nu)^{\oplus r_{\lambda\mu\nu}}.
    \end{displaymath}
\end{problem*}

When $n=0$, this becomes the problem of understanding the restriction of homogeneous polynomial representations of $\glm$ of degree $d$ to $S_m$, which is given by Littlewood's formula ~\cite[Theorem XI]{littlewood1958products}
\begin{align}\label{LittlwoodFormula}
    r_{\lambda\mu\emptyset}=\langle s_{\lambda},s_{\mu}[\tilde H] \rangle_{\Lambda(X)},
\end{align}
where $\Tilde{H}$ is the sum of all monomials in $\{x_{1},x_{2},\dotsc,x_{n}\}$ and $\langle\cdot,\cdot\rangle_{\Lambda(X)}$ is the Hall inner product \cite[Chapter I, Equation 4.5]{macdonald1998symmetric} on the algebra of symmetric functions $\Lambda(X)$.

\begin{remark*}Usually, $H$ denotes the sum of all monomials in $\{x_{1},x_{2},\dotsc,x_{n}\}$. Here, we write this sum as $\tilde{H}$, and reserve $H$ for the corresponding super version (see \eqref{Littlewood}). An analogous convention is adopted for $E$. Similarly, we denote the irreducible supermodule of $S(m,n;d)$ corresponding to partition $\lambda\in\Hi$ by $W_{\lambda}$ and the irreducible module of $S(m,d)$ (equivalently irreducible polynomial representations of $\glm$) corresponding to the partition $\mu=(\mu_{1},\dotsc,\mu_{l})$ ($l\leq d$) by $\tilde{W}_{\lambda}$.
\end{remark*}

The nonnegative integer
$r_{\lambda\mu\emptyset}$ is the \textit{classical} restriction coefficient $r_{\lambda\mu}$ of the Specht module $V_{\mu}$ when the degree $d$ polynomial representation $\tilde{W}_{\lambda}$ of $\glm$ is restricted to $S_{m}$. It is still an open problem to find a combinatorial description of all the classical restriction coefficients, although partial progress can be found in \cite{assaf2020specht,harman2018representations,lee2025restriction,narayanan2021character, narayanan2021polynomial, narayanan2024some,  narayanan2024hook,orellana2021hopf,orellana2021symmetric}.

Let $\Lambda(X|Y)$ denote the algebra of supersymmetric polynomials \cite{stembridge1985characterization} in variables $\{x_{i}\}_{i=1}^{m}\cup\{y_{j}\}_{j=1}^{n}$. The set of hook Schur polynomials $\{hs_{\lambda}\mid\lambda\in\Hi\}$ forms a basis for $\Lambda(X|Y)$. Let $\langle\cdot,\cdot\rangle_{\Lambda(X|Y)}$ denote the inner product with respect to which this basis is orthonormal. We introduce two different versions of plethysm in Section \ref{plethysm} for the space $\Lambda(X|Y)$, where for any symmetric function $f$ and any supersymmetric polynomial $g\in\Lambda(X|Y)$, $f[g]$ and $f\{g\}$ are again elements in $\Lambda(X|Y)$. Our main theorem (Theorem \ref{Littlewood}) generalizes Littlewood's formula:
\begin{align}\label{introLittlewood}
    r_{\lambda\mu\nu} = \langle s_{\mu}[H]s_{\nu}\{E\}, hs_{\lambda}(X,Y) \rangle_{\Lambda(X|Y)},
\end{align}
where
$H=\frac{\prod_{j=1}^{n}(1+y_{i})}{\prod_{i=1}^{m}(1-x_{i})}$ and 
$E=\frac{\prod_{i=1}^{m}(1+x_{i})}{\prod_{j=1}^{n}(1-y_{i})}.$

A crucial tool for the proof of \eqref{introLittlewood} is a \textit{polynomial superinduction} functor $$\ind: S_m \times S_n\rep : \xrightarrow[]{}S(m,n;d)\smod,$$ which is a generalization of the polynomial induction functor defined in \cite{narayanan2021polynomial}. For $\mu\vdash m$ and $\nu\vdash n$, the formal character of the superinduction $\ind (V_{\mu}\otimes V_{\nu})$ is given by (Theorem \ref{inductionFromula})
\begin{align}
    \ch(\ind (V_{\mu}\otimes V_{\nu}))=(s_{\mu}[H]s_{\nu}\{E\})_{d},
\end{align}
where $(s_{\mu}[H]s_{\nu}\{E\})_{d}$ denotes the degree $d$ homogeneous part of the product $s_{\mu}[H]s_{\nu}\{E\}$. 

In Proposition \ref{trivialRestriction}, we provide signed combinatorial formulae for $r_{\lambda\mu\nu}$ for certain $\lambda$'s, when $\mu$ and $\nu$ are trivial or sign representations. Using these formulae, we show the unimodality of \textit{bipartite superpartitions} in Section \ref{bipartiteNumbers} which generalizes a result of Kim and Hahn \cite{kim1997partitions}.

\subsection{Acknowledgements}I am grateful to Amritanshu~Prasad for many valuable discussions and insightful comments during the preparation of this article, and for carefully reading the entire manuscript. I thank Aritra~Bhattacharya for valuable feedback on an earlier draft. I also thank Aritra~Bhattacharya and Nishu~Kumari for their input leading to Remark~\ref{superization}.

\section{Schur superalgebra}
\subsection{Notation} We denote the set of non-negative integers by $\mathbb{N}$, the set of positive integers by $\mathbb{P}$ and the set $\{0,1\}$ by $\mathbb{B}$. Given any partition $\lambda$, denote its length by $l(\lambda)$.
\subsection{Superspaces} Throughout the article, all vector spaces are assumed to be over $\mathbb{C}$. A vector superspace $V$ is a vector space with a $\mathbb{Z}_2$-grading, i.e., a decomposition $V=V_{0}\oplus V_{1}$. Given a homogeneous vector $v\in V$, denote its $\mathbb{Z}_2$-degree/parity by $|v|\in\mathbb{Z}_2$.
When $V$ and $W$ are superspaces, $\Hom_{\mathbb{C}}(V,W)$ of linear maps admits a decomposition
$$
    \Hom_{\mathbb{C}}(V,W) = \Hom_{\mathbb{C}}(V,W)_0 \oplus \Hom_{\mathbb{C}}(V,W)_1,
$$
where
$$
\Hom_{\mathbb{C}}(V,W)_i = \{T\in \Hom_{\mathbb{C}}(V,W)\mid T(V_j)\subset T(V_{j+i})\} \text{ for }i,j\in \mathbb Z/2\mathbb Z.
$$
A subsuperspace of a vector superspace $V$ is a vector subspace $W$ of $V$ such that $W$ has a decomposition $W=W_{0}\oplus W_{1}$ with $W_0\subset V_0$ and $W_1\subset V_1$. Given two vector superspaces $M$ and $N$, one can form their direct sum $M\oplus N$ and tensor $M\otimes N$ by setting:
$(M\oplus N)_{i}=M_{i}\oplus N_{i}$ for $i\in\mathbb{Z}_{2}$; $(M\otimes N)_{0}=(M_{0}\otimes N_{0})\oplus (M_{1}\oplus N_{1})$ and $(M\otimes N)_{1}=(M_{0}\otimes N_{1})\oplus (M_{0}\oplus N_{1})$.
\subsection{Superalgebras} A superalgebra is a superspace $A$ with an additional structure of an associative unital $\mathbb{C}$-algebra such that $A_{i}A_{j}\subseteq A_{i+j}$ for each $i,j\in\mathbb{Z}_{2}$. A superalgebra morphism $f:A\xrightarrow{}B$ is an even linear map which is also an algebra homomorphism.

\subsection{Supermodules}
Let $A$ be a superalgebra. A left supermodule of $A$ is a superspace $M$ that is also an $A$-module such that $A_{i}M_{j}\subseteq M_{i+j}$ for $i,j\in\mathbb{Z}_{2}$. Similarly, one defines right supermodules. A morphism $\phi:M\xrightarrow{}N$ of left $A$-supermodules is a linear map such that $$\phi(a.m)=(-1)^{|a||\phi|}a.\phi(m).$$
A morphism $\phi:M\xrightarrow{}N$ of right $A$-supermodules is a linear map such that $$\phi(m.a)=\phi(m).a.$$
Let $B$ be another superalgebra. We say a superspace $M$ is a $(A,B)$-bisupermodule if $M$ is a left $A$-supermodule and right $B$-supermodule such that
$$(am)b=a(mb)$$ for all $a\in A, b\in B$ and $m\in M$.
\subsection{Schur superalgebra}\label{superPolynomials}
Let $[m]=\{1,2,\dotsc,m\}$, $[\Bar{n}]=\{\Bar{1},\dotsc,\Bar{n}\}$ and $[m|n]=[m]\cup[\Bar{n}]$. Equip $[m|n]$ with the ordering $$1<2<\cdots<m<\Bar{1}<\Bar{2}<\cdots<\Bar{n}.$$ Define the parity function $p:[m|n]\xrightarrow[]{}\mathbb{Z}_{2}$ by setting $p(i)=0$ if $i\in [m]$ and $p({j})=1$ if $j\in [\Bar{n}]$. Consider the vector superspace $\mathbb{C}^{m}\oplus \mathbb{C}^{n}$ where the $0$th (resp. $1$th) graded part is $\mathbb{C}^{m}$ (resp. $\mathbb{C}^{n}$). Denote the superspace $\bigotimes^{d}(\mathbb{C}^{m}\oplus \mathbb{C}^{n})$ by $\T$. A $\mathbb{Z}_2$-homogeneous basis for $\T$ is given by $\{e_{i}\}_{i=1}^{m+n}$ where $e_{i}$ is the $i$th coordinate vector in $\mathbb C^{m+n}$. Let $\sigma_i$ be the simple transposition $(i,i+1)$ for each $i\in[d-1]$. $S_d$ acts on $\T$ \cite[Definition 1.3 and Lemma 4.10]{berele1987hook}  as follows:
\begin{align}\label{symmetricAction}
\sigma_i\cdot (e_{i_{1}}\otimes e_{i_{2}}\otimes\cdots\otimes e_{i_{d}})=
    \begin{cases}
        e_{\sigma_i(1)}\otimes e_{\sigma_i(2)}\otimes\cdots\otimes e_{\sigma_i(d)} &\text{if }p(i)p(i+1)=0\\
        -e_{\sigma_i(1)}\otimes e_{\sigma_i(2)}\otimes\cdots\otimes e_{\sigma_i(d)}&\text{if }p(i)p(i+1)=1.
    \end{cases}
\end{align}

With this action $\T$ becomes a right $\mathbb{C}[S_{d}]$-module. The commutant $\text{End}_{S_{d}}(\T)$ is called the \textit{Schur superalgebra} \cite[Section 3.1]{hemmer2004representation}, and is denoted by $S(m,n;d)$. We give an interpretation of $S(m,n;d)$ as a dual of the superalgebra of polynomials in \textit{signed indeterminates} $(t_{ij})_{i,j=1}^{m+n}$.

Let $A(m,n)$ be the $\mathbb{C}$-algebra generated by $t_{ij}$ for $i,j\in [m|n]$ subject to the relations $$t_{ij}t_{kl}=(-1)^{(p(i)+p(j))(p(k)+p(l))}t_{kl}t_{ij}.$$ Assign each generator the parity $|t_{ij}|=p(i)+p(j)$. This extends uniquely to monomials by
$$\bigg|\prod_{i=1}^{m+n}\prod_{j=1}^{m+n}t_{ij}^{a_{ij}}\bigg|=\sum_{i=1}^{m+n}\sum_{j=1}^{m+n} a_{ij}|t_{ij}|\in\mathbb{Z}_2,$$
where $a_{ij}\in\mathbb N$. Since parity is preserved under multiplication by non-zero scalars, this also defines parity for every homogeneous element of $A(m,n)$. Let 
$$A(m,n)=A(m,n)_0 \oplus A(m,n)_1,$$
where $A(m,n)_{\epsilon}$ is the subspace of $A(m,n)$ spanned by monomials of parity $\epsilon\in\mathbb Z_2$. Hence $A(m,n)$ is a superalgebra. Let $A(m,n;d)$ denote the subspace spanned by homogeneous elements of degree $d$. The subalgebra $A(m,n;d)$ respects the $\mathbb{Z}_2$-grading of $A(m,n)$ and hence forms a subsuperalgebra of $A(m,n)$.
The space $A(m,n;d)$ admits a coalgebra structure coming from matrix multiplication:
\begin{displaymath}
    \Delta t_{ij} = \sum_{k=1}^{m+n} t_{ik}\otimes t_{kj}.
\end{displaymath}
Let $A(m,n;d)^{*}$ denote the resulting superalgebra structure on the linear dual of $A(m,n;d)$.
Then the following theorem is well-known (see e.g., \cite[Section 5]{brundan2003new}).
\begin{theorem} \label{isomorphismSchur}
    There is an isomorphism of superalgebras $S(m,n;d)\cong A(m,n;d)^{*}$ for all $m,n,d\in \mathbb P$.
\end{theorem}
\begin{definition}
    Let $M(m,n)$ denote the set of $(m+n)\times (m+n)$ matrices $A=(a_{ij})$ whose diagonal $m\times m$ and $n\times n$ blocks have entries from $\mathbb N$, and off-diagonal blocks have entries from $\mathbb B$ (see Fig.~\ref{fig:block-matrix}). Given $A\in M(m,n)$, we denote its $i$th row (resp. $j$th column) by $\ai$ (resp. $\aj$). Let $M(m,n;d)\subset M(m,n)$ consists of those matrices in $M(m,n)$ satisfying $$\sum_{1\leq i,j\leq m+n}a_{ij}=d.$$
    Given a matrix $A\in M(m,n)$, we call its off-diagonal blocks the \textit{super part of $A$} (see Figure \ref{fig:block-matrix}). We call the entries of $\ai$ (resp. $\aj$) that lie in the super part of $A$ the super part of $\ai$ (resp. $\aj$). Let $n(\ai)$ (resp. $n(\aj)$) denote the number of $1$'s in the super part of $\ai$ (resp. $\aj$).
\end{definition}

\begin{figure}[h]
\centering
\begin{tikzpicture}
\matrix (m) [matrix of math nodes,
             nodes in empty cells,
             left delimiter={[}, right delimiter={]},
             row sep=6pt, column sep=6pt] {
3 & 12 & 0 & 5 & 8 & 1 & 0 & 1 \\
7 & 0 & 4 & 2 & 6 & 0 & 1 & 0 \\
1 & 9 & 3 & 0 & 5 & 1 & 0 & 1 \\
4 & 6 & 2 & 8 & 1 & 0 & 1 & 0 \\
0 & 5 & 7 & 3 & 9 & 1 & 0 & 1 \\
0 & 1 & 0 & 1 & 1 & 4 & 6 & 1 \\
1 & 0 & 1 & 0 & 1 & 7 & 5 & 0 \\
0 & 1 & 0 & 1 & 1 & 2 & 8 & 6 \\
};
\begin{scope}[on background layer]
\fill[gray!25]
(m-1-6.north west) rectangle (m-5-8.south east);

\fill[gray!25]
(m-6-1.north west) rectangle (m-8-5.south east);
\end{scope}

\end{tikzpicture}

\caption{An element of $M(5,3)$ with the super part shaded.}
\label{fig:block-matrix}

\end{figure}

For any matrix $A\in M(m,n;d)$, define the monomial
\begin{equation}\label{eq:monomial}
    t^A = \prod_{i=1}^{m+n}\prod_{j=1}^{m+n}t_{ij}^{a_{ij}}.
\end{equation}
A $\mathbb{Z}_{2}$-homogeneous basis for $A(m,n;d)$ is given by $\{t^A\mid A\in M(m,n;d)\}$, where recall that the parity of $t^A$ is given by $|t^{A}|=\sum_{i=1}^{m+n}\sum_{j=1}^{m+n}a_{ij}(p(i)+p(j))\in\mathbb{Z}_2$. The set $$\{\xi_{A}\mid A\in M(m,n;d)\},$$ where $\xi_{A}(t^{B})=\delta_{A,B}$ for $A,B\in M(m,n;d)$, forms a $\mathbb{Z}_2$-homogeneous basis for $S(m,n;d)$. Here the parity of $\xi_{A}$ is given by $|t^{A}|$. In \cite{brundan2003new}, the authors use pairs of functions $\mathbf{i}:[d]\xrightarrow[]{}[m+n]$ to index the above bases of $A(m,n;d)$ and $S(m,n;d)$. We now explain the relationship between such pairs of functions and matrices in $M(m,n;d)$. Let $I(m,n;d)$ denote the set of all functions $\mathbf{i}:[d]\xrightarrow[]{}[m+n]$. $S_{d}$ acts on $I(m,n;d)$ as follows:
$$(i_{1},\dotsc,i_{d})\cdot\sigma=(i_{\sigma(1)},\dotsc,i_{\sigma(d)}). $$

A pair $(\mathbf{i},\mathbf{j})\in I(m,n;d)\times I(m,n;d)$ is said to be \textit{strict} if $$p(i_{k})+p(j_{k})=p(i_{l})+p(j_{l})=0$$ whenever $(i_{k},j_{k})=(i_{l},j_{l})$ for $1\leq k<l\leq d$. Let $I^{2}(m,n;d)$ denote the set of strict pairs in $I(m,n;d)\times I(m,n;d)$. Let $\Omega(m,n;d)$ denote the set of orbits for the action of $S_d$ on $I^{2}(m,n;d)$. For a pair $(\mathbf{i},\mathbf{j})\in I(m,n;d)\times I(m,n;d)$, let $\psi(\mathbf{i},\mathbf{j})$ be the matrix with entries $\psi(\mathbf{i},\mathbf{j})_{ij}$ being the number of $(i_{k},j_{k})$ in $(\mathbf{i},\mathbf{j})$ satisfying $(i_{k},j_{k})=(i,j)$. The map $\psi$ preserves the action of $S_d$ on $I(m,n;d)\times I(m,n;d)$ and induces a bijection from $\Omega(m,n;d)$ to $M(m,n;d)$.

\subsection{$S_{m}\times S_{n}$ inside $S(m,n;d)$}
The $0$-graded part of $S(m,n;d)$ is isomorphic to the product $S(m;d)\times S(n;d)$ of classical Schur algebras. Moreover, $S(m;d)$ (resp. $S(n;d)$) contains $S_{m}$ (resp. $S_n$) (see \cite[Example 6.2.1]{prasad2015representation}). Hence we have the following inclusions:
$$S_{m}\times S_{n}\subset S(m;d)\times S(n;d)\subset S(m,n;d).$$ More precisely, for $(\tau,\sigma)\in S_m \times S_n$, the corresponding element in $S(m,n;d)$ is given by $\delta_{(\tau,\sigma)}=\sum_{\text{supp}(A)\subseteq \text{supp}(\tau,\sigma)}\xi_{A},$ where $(\tau,\sigma)$ is interpreted as an $(m+n)\times (m+n)$ block permutation matrix.
Here, $\text{supp}(A)$ denotes the set of pairs $(i,j)$ such that $a_{ij}\neq 0$.
\begin{definition}
    Let $\text{C}(l,d)=\{(a_{1},\dotsc,a_{l})\in\mathbb{N}^{l}\mid\sum_{i=1}^{l}a_{i}=d\}$ be the set of all weak compositions of $d$ with at most $l$ parts. Given $t=(t_{1},t_{2},\dotsc,t_{l})\in \mathbb{N}^{l}$, let $|t|=\sum_{i=1}^{l}t_{i}$.
\end{definition} Let $\xi_{\mu}=\xi_{D(\mu)}$, where $D(\mu)=\text{diag}(\mu_{1},\dotsc,\mu_{m+n})$. Then $\sum_{\mu\in \text{C}(m+n,d)}\xi_{\mu}$ is the identity element in $S(m,n;d)$. Moreover, these summands are mutually orthogonal idempotents.

\subsection{Multiplication in $S(m,n;d)$} In this short section, we record the multiplication rules of certain elements in $S(m,n;d)$ that we use frequently in the text. A general formula for the multiplication of the basis elements in $S(m,n;d)$ can be found in \cite[Lemma~5.1]{brundan2003new}. 

\begin{lemma}\label{multiplicationByTranspositions}
    Assume that $A\in M(m,n;d)$. Let $(ij)$ be the transposition in $S_m$ or $S_n$ that interchanges $i$ and $j$, and $\mathrm{Id}_{m}$ (resp. $\mathrm{Id}_{n}$) be the identity in $S_{m}$ (resp. $S_n$).  
    \begin{enumerate}
        \item Suppose $A_{i*}=A_{j*}$ for some $i,j$ with $1\leq i<j\leq m$, then we have $$\delta_{((ij),\mathrm{Id}_{n})}\xi_{A}=(-1)^{n(A_{i*})}\xi_{((ij),\mathrm{Id}_{n})A}.$$
        \item Suppose $A_{i*}=A_{j*}$ for some $i,j$ with $m+1\leq i<j\leq m+n$, then we have $$\delta_{(\mathrm{Id}_{m},(ij))}\xi_{A}=(-1)^{n(A_{i*})}\xi_{(\mathrm{Id}_{m},(ij))A}.$$
    \end{enumerate}
    Here, $((ij),\mathrm{Id}_{n})A$ and $A(\mathrm{Id}_{m},(ij))$ are matrix products.
\end{lemma}

\begin{lemma}\label{compositionMultiplication}
    For any $\mu\in C(m+n,d)$ and $A\in M(m,n;d)$, we have
        \begin{equation*}
        \xi_{A}\xi_{\mu}=
            \begin{cases}
                \xi_{A} & \text{if $\mu_{i}=|A_{*i}|$ for all $i\in [m+n]$},\\
                0 & \text{otherwise.}
            \end{cases}
        \end{equation*}
\end{lemma}

\begin{lemma}\label{idempotents}{\em \cite[Lemma 5.3]{brundan2003new}} 
$\xi_{\mu}\xi_{\nu}=\delta_{\mu,\nu}\xi_{\mu}$ for $\mu,\nu\in C(m+n,d)$, where $\delta_{\mu,\nu}$ is the Kronecker delta.
\end{lemma}

\subsection{Modules of $S(m,n;d)$} We record here a few facts in the supermodule theory of $S(m,n;d)$. Their proofs can be found in \cite{berele1987hook}. If $\lambda$ is a partition of $d$, we write $\lambda\vdash d$. Let $\lambda^{'}$ denote the patition conjugate to $\lambda$. The simple supermodules of $S(m,n;d)$ are indexed by partitions $\lambda\vdash d$ in the $(m,n)$-hook i.e., $\lambda_{m+1}\leq n$. Let $\Hi$ denote the set of all partitions of $d$ in the $(m,n)$-hook. Let $W_{\lambda}$ denote the simple $S(m,n;d)$-module corresponding to $\lambda\in\Hi$. The dimension of $W_{\lambda}$ is given by the number of \textit{semistandard supertableaux} in $[m|n]$ of shape $\lambda$ which we define below. Recall that $[m|n]$ is equipped with the ordering $$1<2<\cdots<m<\Bar{1}<\Bar{2}<\cdots<\Bar{n}.$$
\begin{definition}\cite[Definition 2.1]{berele1987hook}
    A semistandard supertableau of shape $\lambda$ is a filling of the boxes in the Young diagram of $\lambda$ using entries from $[m|n]$ such that
    \begin{enumerate}
        \item The entries are weakly increasing along rows and columns.
        \item The $\Bar{i}$'s from $[\Bar{n}]$ (resp. $j$'s from $[m]$) are strictly increasing along rows (resp. columns).
    \end{enumerate}
\end{definition}
Denote the set of all such tableaux by $\text{SSYT}(\lambda,m,n)$.
Define the weight of a box containing $i$ (resp. $i'$) in a supertableau to be $x_i$ (resp. $y_i$).
The weight $(x,y)^T$ of the supertableau $T$ is defined to be the product of the weights of its cells.
\begin{example}
$T_{1}$ is a semistandard supertableau while $T_{2}$ is not.

\[
\begin{array}{c@{\hspace{2cm}}c}
\begin{ytableau}
1 & 1 & \Bar{2} & \Bar{3} \\
2 & 3 & \Bar{2}\\
3\\
\Bar{4}
\end{ytableau}
&
\begin{ytableau}
1 & \Bar{2} & \Bar{2} \\
2 & 3 & \Bar{3}\\
\Bar{4}
\end{ytableau}
\\[6pt]
T_{1} & T_{2}
\end{array}
\]
We have $(x,y)^{T_1}=x_1^2x_2x_3y_2^2y_3y_4$.
\end{example}  

\begin{definition}[Supersymmetric polynomials]\label{supersymmetricfunctions}
Let $\Lambda(X)=\Lambda[x_{1},\dotsc,x_{m}]$ (resp. $\Lambda$) denote the algebra of symmetric polynomials (resp. symmetric functions) over the field $\mathbb{C}$. Let $$\Lambda(X,Y)=\Lambda[x_{1},\dotsc,x_{m}]\otimes_{\mathbb{C}}\Lambda[y_{1},\dotsc,y_{n}]$$ and $\Lambda(X|Y)\subset \Lambda(X,Y)$ denote the subalgebra consisting of polynomials $$f(x_{1},\dotsc,x_{m};y_{1},\dotsc,y_{n})$$ satisfying the condition that the polynomial $f(u,x_{2},\dotsc,x_{m};-u,y_{2},\dotsc,y_{n})$ is independent of $u$. Elements in $\Lambda(X|Y)$ are called supersymmetric polynomials. \cite{stembridge1985characterization} 
\end{definition}

\begin{remark}
    This definition is equivalent to the classical definition of supersymmetric polynomials given in \cite{stembridge1985characterization} and \cite{macdonald1998symmetric} under the coordinate change $y_{j}\mapsto -y_{j}$ for all $j\in [n]$.
\end{remark}

Let $\Lambda(X|Y)_{d}$ (resp. $\Lambda(X,Y)_d$) denote the subspace consisting of degree $d$ homogeneous elements in $\Lambda(X|Y)$ (resp. $\Lambda(X,Y)$). We denote $f(x_{1},\dotsc,x_{m};y_{1},\dotsc,y_{n})\in \Lambda(X,Y)$ (resp. $f(x_{1},\dotsc,x_{m})\in\Lambda(X)$) simply by $f(X,Y)$ (resp. $f(X)$). Given a polynomial $f(X,Y)\in \Lambda(X,Y)$, write $f(X,Y)_{d}$ to denote the degree $d$ homogenous part of $f(X,Y)$.
\begin{definition}(Hook Schur polynomials)\cite[Definition 6.3]{berele1987hook}\label{hookSchurdef}     Let $\lambda\in\Hi$. The hook Schur polynomial $hs_{\lambda}(X,Y)$ is defined by
$$hs_{\lambda}(X,Y)=\sum_{|\mu|+|\nu|=d}c^{\lambda}_{\mu,\nu'}s_{\mu}(X)s_{\nu}(Y).$$
\end{definition}

 A description of $hs_{\lambda}$ as a sum of weights of semistandard supertableaux is given in the following proposition.
\begin{proposition}\label{tableauxformula}{\em\cite[Theorem 6.10]{berele1987hook}}
    Let $\lambda\in\Hi$. The hook Schur polynomial $hs_{\lambda}(X,Y)$ is given by
    $$hs_{\lambda}(X,Y)=hs_{\lambda}(x_{1},\dotsc,x_{m};y_{1},\dotsc,y_{n})=\sum_{T\in \text{SSYT}(\lambda,m,n)}(x,y)^{T}.$$
\end{proposition}
\begin{theorem}{\em\cite[Chapter I.3, Example 24.(c)]{macdonald1998symmetric}}\label{hookSchurBasis}
    Let $\lambda\in\Hi$. Then $hs_{\lambda}$ is supersymmetric. Moreover, the set
    $$\{hs_{\lambda}(X,Y)\mid\lambda\in\Hi\}$$ forms a basis for $\Lambda(X|Y)_{d}$.
\end{theorem}

\begin{definition}
    Let $\langle\cdot,\cdot\rangle_{\Lambda(X|Y)}$ (resp. $\langle\cdot,\cdot\rangle_{\Lambda(X,Y)}$) denote the inner product in $\Lambda(X|Y)$ (resp. $\Lambda(X,Y)$) with respect to which the basis $$\{hs_{\lambda}\mid \lambda\in\Hi\}\text{ } (\text{resp. }\{s_{\lambda}(X)s_{\mu}(Y)\mid\lambda,\mu\in \text{Par}\})$$ is orthonormal.
\end{definition}

\begin{example}
 \begin{enumerate}
    \item The power sum supersymmetric polynomial is defined by
    $$hp_i (X,Y) = \sum_{j=1}^{m}x_{j}^{i}+(-1)^{i-1}\sum_{k=1}^{n}y_{k}^{i}.$$
     \item If $\lambda=(d)$, then using Definition \ref{hookSchurdef}, we have
     \begin{align*}
         hs_{(d)}(X,Y)=\sum_{i=0}^{d}h_{i}(X)e_{d-i}(Y).
     \end{align*}
     We call $hs_{(d)}(X,Y)$ the $d$th complete supersymmetric polynomial and denote it by $h_{d}(X,Y)$. 
     \item If $\lambda=(1^{d})$, again using Definition \ref{hookSchurdef},
     we have
     \begin{align*}
         hs_{(1^{d})}=\sum_{i=0}^{d}e_{i}(X)h_{d-i}(Y).
     \end{align*}
     We call $hs_{(1^{d})}(X,Y)$ the $d$th elementary supersymmetric polynomial and denote it by $e_{d}(X,Y)$. 
 \end{enumerate}
\end{example}

\subsection{Formal characters of $S(m,n;d)$-modules} In this section, we give a brief account of the formal character theory for $S(m,n;d)$. More details can be found in \cite[Section 6.5]{axtell2018schur}.  Let $V$ be a finite dimensional $S(m,n;d)$-module. Consider the matrix $D(x,y)=\text{diag}(x_{1},\dotsc,x_{m},y_{1},\dotsc,y_{n})$ where $x_{i}$'s and $y_{j}$'s are formal symbols. For $\mu\in \mathbb{N}^{m+n}$, define $$(x,y)^{\mu}=\prod_{i=1}^{m}x_{i}^{\mu_{i}}\prod_{j=1}^{n}y_{j}^{\mu_{m+j}}.$$ Define the element $$\delta_{x,y}=\delta_{D(x,y)}=\sum_{\mu\in\text{C}(m+n,d)}(x,y)^{\mu}\xi_{\mu}$$ inside $S(m,n;d)\otimes_{\mathbb{C}} \mathbb{C}[x_{1},\dotsc,x_{m},y_{1},\dotsc,y_{n}]$. Since $\sum_{\mu\in\text{C}(m+n,d)}\xi_{\mu}$ is a decomposition of the identity element of $S(m,n;d)$ into pairwise orthogonal idempotents (Lemma \ref{idempotents}), we have the weight space decomposition $$V=\bigoplus_{\mu\in\text{C}(m+n,d)}\xi_{\mu}V$$
for any $S(m,n;d)$-supermodule $V$. The formal character of $V$ is defined as
$$\ch(V)=\text{trace}(\delta_{x,y};V)=\sum_{\mu\in\text{C}(m+n,d)}\text{trace}(\xi_{\mu};V)\cdot(x,y)^{\mu}=\sum_{\mu\in\text{C}(m+n,d)}\text{dim}(\xi_{\mu}V)\cdot(x,y)^{\mu}.$$  The last equality follows from the fact that the idempotent $\xi_{\mu}$ acts on $\xi_{\mu}V$ by identity and on $\xi_{\nu}V$ by $0$ if $\nu\neq\mu$.
\begin{proposition}{\em \cite[Proposition 6.5.2]{axtell2018schur}}
    For each $\lambda\in\Hi$, we have
    $$\ch(W_{\lambda})=hs_{\lambda}(X,Y).$$
\end{proposition}

\subsection{Superplethysm}\label{plethysm} Recall~\cite[Section 2.1]{carlsson2018proof} that a $\lambda$-ring is a ring $R$ with a collection of ring homomorphisms $(p_{i}:R\xrightarrow[]{}R)_{i\in\mathbb{P}}$ satisfying the following conditions:
\begin{enumerate}
    \item $p_{n}[p_{m}[x]]=p_{nm}[x]$ for all $x\in R$ and $m,n\in\mathbb{P}$.
    \item $p_{1}[x]=x$ for all $x\in R$.
\end{enumerate}
For any partition $\mu=1^{a_{1}}2^{a_{2}}\cdots l^{a_{l}}$ written in \emph{exponential notation} (i.e., the number of times $i$ appears in $\mu$ is $a_{i}$), $p_\mu[x]$ is defined as $\prod_i p_i[x]^{a_i}$, and for any $f\in \Lambda$, 
$f[x]$ is defined as $\sum_\mu c_\mu p_\mu[x]$, where $f=\sum_\mu c_\mu p_\mu$ is the expansion of $f$ in the power sum basis. More generally, for any formal power series $F(v)=\sum_{i\geq 0}f_{i}v^{i}\in \Lambda[[v]]$ over $\Lambda$, where $f_{i}\in\Lambda$ for all $i\geq 0$, define
$$F(v)[x]=\sum_{i\geq 0}f_{i}[x]v^{i}.$$

The $\lambda$-ring structure on $\mathbb{C}[x_{1},\dotsc,x_{m},y_{1},\dotsc,y_{n}]$  defined by
\begin{enumerate}
    \item $p_{i}[x]=x^{i}$ (resp. $p_{i}\{x\}=(-1)^{i-1}x^{i}$) for $x\in\{x_{1},x_{2},\dotsc\}$
    \item $p_{i}[y]=(-1)^{i-1}y^{i}$ (resp. $p_{i}\{y\}=y^{i}$) for $y\in \{y_{1},y_{2},\dotsc\}$
    \item $p[a]=a$ (resp. $p\{a\}=a$) for $a\in \mathbb{C}$
\end{enumerate}
for all $i>0$, restricts to a $\lambda$-ring structure on $\Lambda(X|Y)$ such that the following holds:
$$p_{i}[hp_{j}]=hp_{ij}$$ for $i,j\in\mathbb{P}$.
For any $f\in \Lambda$ and $g\in \Lambda(X|Y)$, the resulting supersymmetric polynomials $f[g]$ and $f\{g\}$ are supersymmetric versions of plethysm that appears in the generalization \eqref{introLittlewood} of Littlewood's formula.
\begin{remark}\label{superization}
    The plethysm $f[g]$, where $f\in\lambda, g\in\Lambda(X|Y)$), is induced by the usual plethysm on symmetric polynomials via the superization map \cite[Definition 4.1]{haglund2005combinatorial}.
\end{remark}
\section{Superrestriction coefficients and polynomial superinduction}\label{superInduction}
The left regular action of $S(m,n;d)$ restricts to a left action of $S_{m}\times S_{n}$ on $S(m,n;d)$. Thus $S(m,n;d)$ is a $(\mathbb{C}[S_{m}\times S_{n}],S(m,n;d))$-bisupermodule.
\begin{definition}
    Define the superinduction functor $S_m\times S_n\rep\to S(m,n;d)\smod$ by
    \begin{displaymath}
        \ind V = \Hom_{S_m\times S_n}(S(m,n;d), V).
    \end{displaymath}
\end{definition}
For $f\in \Hom_{\mathbb C}(S(m,n;d), V)$, we write
\begin{displaymath}
    f = \sum_{A\in M(m,n;d)} f_A\bullet\xi_A,
\end{displaymath}
where $f_A = f(\xi_A)\in V$.
Then
\begin{displaymath}
    \ind V = \Big\{\sum_{A\in M(m,n;d)} f_A\bullet\xi_A \mid (\sigma,\tau)\cdot f_{A}=\sign(A,\sigma,\tau)f_{(\sigma,\tau)A}\Big\},
\end{displaymath}
where $\sign(A,\sigma,\tau)$ is defined by requiring $$\delta_{(\sigma,\tau)}\xi_{A}=\sign(A,\sigma,\tau)\xi_{(\sigma,\tau)A}.$$ Here $(\sigma,\tau)\in S_{m}\times S_{n}$ is interpreted as a $(m+n)\times(m+n)$ block diagonal permutation matrix, and $(\sigma,\tau)A$ is the matrix product.
It follows from Lemma~\ref{multiplicationByTranspositions} that $\sign(A,\sigma,w)=\pm 1$. The action of $S(m,n;d)$ on $\ind V$ \cite[Equation 3.4]{rosso2015towers} is given by
$$\xi\cdot f= \sum_{A\in M(m,n;d)} (-1)^{|\xi|(|f|+|\xi_{A}|)}f(\xi_{A}\xi)\bullet\xi_A,$$
for all homogeneous elements $\xi\in S(m,n;d)$ and  $f\in\ind V$; $|f|\in\mathbb{Z}_2$ denotes the homogeneous degree of $f$. 

For any $S(m,n;d)$-supermodule $W$, let $\res W$ denote the $S_{m}\times S_{n}$-representation obtained by restricting the action of $S(m,n;d)$ on $W$ to the multiplicative subgroup $S_{m}\times S_{n}$.
In other words,
\begin{displaymath}
    \res W = S(m,n;d)\otimes_{S(m,n;d)}W.
\end{displaymath}
\begin{proposition}[Frobenius reciprocity]\label{frobeniusReciprocity}
    For every representation $(\rho,U)$ of $S_{m}\times S_n$ and supermodule $(\omega,V)$ of $S(m,n;d)$, there is an even natural isomorphism:
    $$\Hom_{S(m,n;d)}(V,\ind U)\cong\Hom_{S_{m}\times S_{n}}(\res V,U).$$
\end{proposition}
\begin{proof}
    This is a special case of the super version of the hom-tensor adjunction (see e.g., \cite{rosso2015towers, deligne1999notes}).
    Define the maps
    $$\phi: \Hom_{S(m,n;d)}(V,\ind U)\xrightarrow[]{}\Hom_{S_{m}\times S_{n}}(\res V,U)$$
    and 
    $$\psi:\Hom_{S_{m}\times S_{n}}(\res V,U)\xrightarrow[]{}\Hom_{S(m,n;d)}(V,\ind U)$$ as follows:
    $$\phi(f)(\xi\otimes_{S}v):=(-1)^{|\xi||v|}f(v)(\xi)$$
    and
    $$\psi(g)(v)(\xi):=(-1)^{|\xi||v|}g(\xi\otimes_{S}v)$$
    for all $f\in \Hom_{S(m,n;d)}(V,\ind U), g\in \Hom_{S_{m}\times S_{n}}(\res V,U)$, and all homogeneous elements $\xi\in S(m,n;d)$ and $v\in V$. It is a straightforward, albeit tedious, calculation to verify that both maps are well-defined, linear, even, and mutual inverses. The result then follows.   
\end{proof}

\begin{definition}
Let $\alpha\in \mathbb{N}^{m}$, $\beta\in \mathbb{N}^{n}$ and $d\in\mathbb{P}$. Let $P(\alpha,\beta;d)$ denote the subset of matrices $A=(a_{ij})$ in $M(m,n;d)$ such that
\begin{enumerate}
        \item $A_{i*}\leq A_{(i+1)*}$ (in lexicographic order) for $i\in\{1,\dotsc,m-1\}\cup \{m+1,\dotsc,m+n-1\}$. 
        \item $|A_{*j}|=\alpha_{j}$ if $1\leq j\leq m$.
        \item $|A_{*j}|=\beta_{j}$ if $1\leq j\leq n$.
\end{enumerate}
Let $M(\alpha,\beta;d)$ denote the subset of $M(m,n;d)$ consisting of matrices $A$ satisfying conditions $(2)$ and $(3)$ above.
\end{definition}
\begin{definition}\label{refinedSuperMatrices}
    Let $\alpha\in \mathbb{N}^{m}$, $\beta\in \mathbb{N}^{n}$ and $d\in\mathbb{P}$.
    \begin{enumerate}
        \item\label{def1} Let $P^{\eta_{m}}_{\eta_{n}}(\alpha,\beta;d)$ (resp. $M^{\eta_{m}}_{\eta_{n}}(\alpha,\beta;d)$) denote the set of matrices $A$ in $P(\alpha,\beta;d)$ (resp. $M(\alpha,\beta;d)$)  such that:
            for $1\leq i< j\leq m$ or $m+1\leq i< j\leq m+n$, if $A_{i*}=A_{j*}$, then $n(A_{i*})=n(A_{j*})$ is even.
        \item Let $P^{\eta_{m}}_{\epsilon_{n}}(\alpha,\beta;d)$ (resp. $M^{\eta_{m}}_{\epsilon_{n}}(\alpha,\beta;d)$) denote the set of matrices $A$ in $P(\alpha,\beta;d)$ (resp. $M(\alpha,\beta;d)$)  such that: for $1\leq i< j\leq m$ (resp. $m+1\leq i<j\leq m+n$), if $A_{i*}=A_{j*}$, then $n(A_{i*})=n(A_{j*})$ is even (resp. $n(A_{i*})=n(A_{j*})$ is odd).
        \item Let $P^{\epsilon_{m}}_{\eta_{n}}(\alpha,\beta;d)$ (resp. $M^{\epsilon_{m}}_{\eta_{n}}(\alpha,\beta;d)$) denote the set of matrices $A$ in $P(\alpha,\beta;d)$ (resp. $M(\alpha,\beta;d)$) such that: for $1\leq i< j\leq m$ (resp. $m+1\leq i<j\leq m+n$), if $A_{i*}=A_{j*}$, then $n(A_{i*})=n(A_{j*})$ is odd (resp. $n(A_{i*})=n(A_{j*})$ is even).
        \item Let $P^{\epsilon_{m}}_{\epsilon_{n}}(\alpha,\beta;d)$ (resp. $M^{\epsilon_{m}}_{\epsilon_{n}}(\alpha,\beta;d)$) denote the set of matrices $A$ in $P(\alpha,\beta;d)$ (resp. $M(\alpha,\beta;d)$)  such that: for $1\leq i< j\leq m$ or $m+1\leq i< j\leq m+n$, if $A_{i*}=A_{j*}$, then  $n(A_{i*})=n(A_{j*})$ is odd.
    \end{enumerate}
\end{definition}

\begin{proposition}\label{trivial}
    For $m,n,d\geq 1$, we have
    $$\ch(\ind (\zeta\otimes\theta))=\sum_{\substack{(\alpha,\beta)\in\mathbb{N}^{m}\times\mathbb{N}^{n}\\|\alpha|+|\beta|=d}}|P^{\zeta}_{\theta}(\alpha,\beta;d)|x^{\alpha}y^{\beta}$$
    where $\zeta\in\{\eta_{m},\epsilon_{m}\}$ and $\theta\in\{\eta_{n},\epsilon_{n}\}$.
\end{proposition}
\begin{proof}
We prove the identity when $\zeta=\eta_{m}$ and $\theta\in\{\eta_{n},\epsilon_{n}\}$. A similar approach works for the remaining cases. 

The element $f= \sum_{A\in M(m,n;d)}f_{A}\bullet\xi_{A}$, where $f_{A}\in\mathbb{C}$, is in  $\ind(\eta_{m}\otimes \eta_{n})$ if and only if
\begin{equation}\label{firstcompatibility}
    f_{(\tau,\sigma)A}=\sign(A,\tau,\sigma)f_{A}
\end{equation}
for every $(\tau,\sigma)\in S_m\times S_n$. Assume that $A_{j*}=A_{k*}$ for some $1\leq j<k\leq m$ with odd $n(A_{j*})=n(A_{k*})$, then,  writing $(jk)$ for the transposition that interchanges $j$ and $k$, $$\sign(A,(jk),\mathrm{Id}_{n})=(-1)^{n(A_{j*})}=-1$$ by Lemma \ref{multiplicationByTranspositions}. In that case $$f_{A}=-f_{((jk),\mathrm{Id}_{n})A}=-f_{A}$$ or, $$f_{A}=0.$$ Similarly, if $A_{j*}=A_{k*}$ for some $m+1\leq j<k\leq m+n$ with odd $n(A_{j*})=n(A_{k*})$, then
$$\sign(A,\mathrm{Id}_{m},(jk))=(-1)^{n(A_{j*})}=-1$$ 
by Lemma \ref{multiplicationByTranspositions} which again implies that $f_{A}=0$. Let $(\alpha,\beta)\in\mathbb{N}^{m}$ and $|\alpha|+|\beta|=d$. Assume that $A\in M^{\eta_{m}}_{\eta_{n}}(\alpha,\beta;d)$. Using \eqref{firstcompatibility}, there exists $(\sigma,\tau)\in S_{m}\times S_n$ and a unique $\tilde{A}\in P^{\eta_{m}}_{\eta_{n}}(\alpha,\beta;d)$ satisfying
$$f_{\tilde{A}}=f_{(\sigma,\tau)A}=\pm f_{A}.$$
Therefore, $\ind(\eta_{m}\otimes \eta_{n})$ has basis indexed by $\bigcup_{\substack{(\alpha,\beta)\in\mathbb{N}^{m}\times\mathbb{N}^{n}\\|\alpha|+|\beta|=d}}P^{\eta_{m}}_{\eta_{n}}(\alpha,\beta;d)$. Using Lemma~\ref{compositionMultiplication}, note that $1\bullet \xi_{A}$ is an eigenvector for $\delta_{x,y}$ with eigenvalue $$\prod_{i=1}^{m}x_{i}^{\alpha_{i}}\prod_{j=1}^{n}{y_{j}}^{\beta_{j}},$$ for any $A\in P^{\eta_{m}}_{\eta_{n}}(\alpha,\beta;d)$. Therefore, the coefficient of $x^{\alpha}y^{\beta}$ in $\ch(\ind(\eta_{m}\otimes \epsilon_{n}))$ is given by $|P^{\eta_{m}}_{\eta_{n}}(\alpha,\beta;d)|$. This proves the proposition when $\zeta=\eta_{m}$ and $\theta=\eta_{n}$.

The element $f = \sum_{A\in M(m,n;d)}f_{A}\bullet\xi_{A}$, where $f_{A}\in\mathbb{C}$, is in $\ind(\eta_{m}\otimes \epsilon_{n})$ if and only if
\begin{equation}\label{secondcompatibility}
    f_{(\tau,\sigma)A}=\sign(A,\tau,\sigma)\sign(\sigma)f_{A}
\end{equation}
for every $(\tau,\sigma)\in S_m\times S_n$. Assume that $A_{j*}=A_{k*}$ for some $1\leq j<k\leq m$ with odd $n(A_{j*})=n(A_{k*})$, then

$$\sign(A,(jk),\mathrm{Id}_{n})=(-1)^{n(A_{j*})}=-1$$
by Lemma \ref{multiplicationByTranspositions}. In that case $$f_{A}=-f_{((jk),\mathrm{Id}_{n})A}=-f_{A}$$ or, $$f_{A}=0.$$ On the other hand, if $A_{j*}=A_{k*}$ for some $m+1\leq j<k\leq m+n$ with even  $n(A_{j*})=n(A_{k*})$, then $$\sign((jk))\sign(A,\mathrm{Id}_{m},(jk))=-\sign(A,\mathrm{Id}_{m},(jk))=-(-1)^{n(A_{*j})}=-1$$ by Lemma \ref{multiplicationByTranspositions}, which implies that $f_{A}=0$. Let $(\alpha,\beta)\in\mathbb{N}^{m}$ and $|\alpha|+|\beta|=d$. Assume that $A\in M^{\eta_{m}}_{\epsilon_{n}}(\alpha,\beta;d)$. Using \eqref{firstcompatibility}, there exists $(\sigma,\tau)\in S_{m}\times S_n$ and a unique $\tilde{A}\in P^{\eta_{m}}_{\epsilon_{n}}(\alpha,\beta;d)$ satisfying
$$f_{\tilde{A}}=f_{(\sigma,\tau)A}=\pm f_{A}.$$
Therefore, $\ind(\eta_{m}\otimes \epsilon_{n})$ has basis indexed by $\bigcup_{\substack{(\alpha,\beta)\in\mathbb{N}^{m}\times\mathbb{N}^{n}\\|\alpha|+|\beta|=d}}P^{\eta_{m}}_{\epsilon_{n}}(\alpha,\beta;d)$. Using Lemma~\ref{compositionMultiplication}, note that $1\bullet\xi_{A}$ is an eigenvector for $\delta_{x,y}$ with eigenvalue $$\prod_{i=1}^{m}x_{i}^{\alpha_{i}}\prod_{j=1}^{n}{y_{j}}^{\beta_{j}},$$ for any $A\in P^{\eta_{m}}_{\epsilon_{n}}(\alpha,\beta;d)$. Therefore, the coefficient of $x^{\alpha}y^{\beta}$ in $\ch(\ind(\eta_{m}\otimes \epsilon_{n}))$ is given by $|P^{\eta_{m}}_{\epsilon_{n}}(\alpha,\beta;d)|$. This proves the proposition when $\zeta=\eta_{m}$ and $\theta=\epsilon_{n}$.
\end{proof}
Let $$H=\frac{\prod_{j=1}^{n}(1+y_{i})}{\prod_{i=1}^{m}(1-x_{i})}=\sum_{s\in\mathbb{N}^{m},\text{ }t\in \mathbb{B}^{n}} x^{s}y^{t}=\sum_{i\geq 0}h_{i}(X,Y)$$ and 
$$E=\frac{\prod_{i=1}^{m}(1+x_{i})}{\prod_{j=1}^{n}(1-y_{i})}=\sum_{t\in\mathbb{N}^{n},\text{ }s\in \mathbb{B}^{m}} x^{s}y^{t}=\sum_{i\geq 0}e_{i}(X,Y).$$

\begin{figure}[h]
\centering

\begin{minipage}{0.45\textwidth}
\centering
\begin{tikzpicture}[baseline=(m.center)]
\matrix (m) [matrix of math nodes,
             nodes in empty cells,
             left delimiter={[}, right delimiter={]},
             row sep=6pt, column sep=6pt] {
3 & 12 & 0 & 5 & 8 & 1 & 0 & 1 \\
7 & 0 & 4 & 2 & 6 & 0 & 1 & 0 \\
7 & 9 & 3 & 0 & 5 & 1 & 0 & 1 \\
7 & 9 & 3 & 0 & 5 & 1 & 0 & 1 \\
};

\begin{scope}[on background layer]
\fill[gray!25]
(m-1-6.north west) rectangle (m-4-8.south east);
\end{scope}
\end{tikzpicture}

\vspace{0.5em}
\[
A
\]
\end{minipage}
\hfill
\begin{minipage}{0.45\textwidth}
\centering
\begin{tikzpicture}[baseline=(m.center)]
\matrix (m) [matrix of math nodes,
             nodes in empty cells,
             left delimiter={[}, right delimiter={]},
             row sep=6pt, column sep=6pt] {
0 & 1 & 0 & 1 & 1 & 4 & 6 & 1 \\
1 & 0 & 1 & 0 & 1 & 7 & 5 & 0 \\
1 & 1 & 0 & 1 & 1 & 2 & 8 & 6 \\
};

\begin{scope}[on background layer]
\fill[gray!25]
(m-1-1.north west) rectangle (m-3-5.south east);
\end{scope}
\end{tikzpicture}

\vspace{0.5em}
\[
B
\]
\end{minipage}

\caption{Matrices $A\in\tilde{P}_{4}(5,3;(24,30,10,7,24),(3,1,3))$ and $B\in\tilde{Q}_{3}(5,3;(2,2,1,2,3),(13,19,7))$, with the super parts shaded.}
\label{fig:AB-blocks1}

\end{figure}

\begin{definition}
    Let $\tilde{P}_{k}(m,n;\alpha,\beta)$ (resp. $\tilde{Q}_{k}(m,n;\alpha,\beta)$) denote the set of $k\times (m+n)$ integer matrices $A$ such that
    \begin{enumerate}
        \item The rows are weakly increasing with respect to the lexicographic order.
        \item $a_{ij}\in \mathbb{B}$ if $j\geq m+1$ (resp. if $j\leq m$) (the submatrix $(a_{ij})_{j\geq m+1}$ (resp. $(a_{ij})_{j\leq m}$) is called the \textit{super part} of $A$).
        \item The vector of column sums is given by $(\alpha,\beta)=(\alpha_{1},\dotsc,\alpha_{m},\beta_{1},\dotsc,\beta_{n})$.
        \item If the rows $A_{i*}$ and $A_{j*}$ are equal then the number of $1$'s in the super part of $A_{i*}$ is even i.e., $n(A_{i*})$ is even.
    \end{enumerate}
    Figure~\ref{fig:AB-blocks1} illustrates the above definition for both
    $\tilde{P}_{k}(m,n;\alpha,\beta)$ and $\tilde{Q}_{k}(m,n;\alpha,\beta)$.
\end{definition}

\begin{lemma}\label{superIdentity}For $k>0$, we have
    \begin{enumerate}
        \item     $h_{k}[H]=\sum_{{\alpha\in\mathbb{N}^{m}, \beta\in\mathbb{N}^{n}}}|\tilde{P}_{k}(m,n;\alpha,\beta)|x^{\alpha}y^{\beta}$.
        \item     $h_{k}\{E\}=\sum_{{\alpha\in\mathbb{N}^{m}, \beta\in\mathbb{N}^{n}}}|\tilde{Q}_{k}(m,n;\alpha,\beta)|x^{\alpha}y^{\beta}$.
    \end{enumerate}
\end{lemma}
\begin{proof}
Let $\tilde H(v):=\sum_{i\geq 0}h_i v^{i}\in \Lambda[[v]]$. Then
\begin{equation}\label{superH}
    \tilde{H}(v)[H] = \sum_{d\geq 0}h_{d}[H]v^{d}.
\end{equation}
 Using the power sum expansion of complete symmetric functions \cite[Equation~2.14]{macdonald1998symmetric}, we have
\begin{displaymath}
    \tilde H (v) = \sum_\lambda (p_\lambda v^{|\lambda|})/z_\lambda,
\end{displaymath}
the sum being over all integer partitions, and $z_\lambda$ is the order of the centralizer of a permutation with cycle type $\lambda$.
Writing $\lambda=1^{a_1}2^{a_2}\dotsb$ in exponential notation gives
\begin{displaymath}
    \tilde H (v) = \prod_{i\geq 1}\sum_{a_i\geq 0} \frac{p_i^{a_i}v^{ia_{i}}}{i^{a_i}a_i!},
\end{displaymath}
whence
\begin{align*}
    \tilde H (v)[H] & = \prod_{i\geq 1}\sum_{a_i\geq 0} (p_i[H]v^{i})^{a_i} /(i^{a_i}a_i!)\\
    & = \prod_{i\geq 1}\exp (p_i[H] v^{i}/i)\\
    & = \exp\big(\sum_{i\geq 1} p_i[H] v^{i}/i\big)\\
    & = \exp\big(\sum_{i\geq 1} \sum_{s\in \mathbb N^m,\;t\in \mathbb B^n} (-1)^{(i-1)|t|}(x^s y^t v)^i/i\big)\\
    & = \exp\big(\sum_{s\in \mathbb N^m}\sum_{t\in \mathbb  B^n, |t| \text{ odd}}\sum_{i\geq 1}(-1)^{i-1}(x^s y^t v)^i/i + \sum_{s\in \mathbb N^m}\sum_{t\in \mathbb  B^n, |t| \text{ even}}\sum_{i\geq 1}(x^s y^t v)^i/i\big)\\
    & = \exp\big(\sum_{s\in \mathbb N^m}\sum_{t\in \mathbb  B^n, |t| \text{ odd}} \log(1+x^sy^t v)+\sum_{s\in \mathbb N^m}\sum_{t\in \mathbb  B^n, |t| \text{ even}}\log(1-x^sy^t v)^{-1}\big)\\
    & = \bigg(\prod_{s\in \mathbb N^m}\prod_{t\in \mathbb  B^n, |t| \text{ odd}}(1+x^sy^t v)\bigg)\bigg(\prod_{s\in \mathbb N^m}\prod_{t\in \mathbb  B^n, |t| \text{ even}}(1-x^sy^t v)^{-1}\bigg)\\
    & = \sum_{k\geq 0}\sum_{\alpha\in \mathbb N^m,\;\beta\in \mathbb N^n}|\tilde P_k(m,n;\alpha,\beta)|x^\alpha y^\beta v^{k}.
\end{align*}

Comparing the coefficients of $v^k$ in the above equation with those in \eqref{superH}, we obtain the first identity. The proof of the second identity is analogous.
\end{proof}
\begin{proposition}\label{lem2} For any $m,n\geq 0$ and $d>0$, we have
$$\ch(\ind(\eta_{m}\otimes \eta_{n}))=(h_{m}[H]h_{n}\{E\})_{d}.$$
\end{proposition}
\begin{proof}
Using Lemma \ref{superIdentity}, we obtain
\begin{align}\label{1}
    h_{m}[H]=\sum_{\alpha\in\mathbb{N}^{m},\text{ }\beta\in\mathbb{N}^{n}}|\tilde{P}_{m}(m,n;\alpha,\beta)|x^{\alpha}y^{\beta}
\end{align} and 
\begin{align}\label{2}
    h_{n}\{E\}=\sum_{\alpha\in\mathbb{N}^{m},\text{ } \beta\in\mathbb{N}^{n}}|\tilde{Q}_{n}(m,n;\alpha,\beta)|x^{\alpha}y^{\beta}.
\end{align}
Multiplying \eqref{1} and $\eqref{2}$ and using Proposition \ref{trivial}, we obtain the result. 
\end{proof}

Assume that $S_m$ and $S_n$ act on the sets $U$ and $V$ respectively. Then $S_{m}\times S_n$ acts on $U\times V$ and one can consider the permutation representation $\mathbb{C}[U\times V]\cong \mathbb{C}[U]\otimes \mathbb{C}[V]$ of $S_{m}\times S_{n}$. 
If the actions of $S_{m}$ and $S_{n}$ on $U$ and $V$ are transitive, the action of $S_{m}\times S_{n}$ on $U\times V$ is also transitive.  We have the following description of $\ind \mathbb{C}[U\times V]$.

\begin{lemma}\label{lemma:perm-induction}
    Assume that there are transitive actions of $S_{m}$ on $U$ and of $S_{n}$ on $V$.
    Fix $(u_{0},v_{0})\in U\times V$ and let $\stab(u_0, v_0)$ denote the stabilizer of $(u_0,v_0)$ for the action of $S_{m}\times S_n$ on $\mathbb{C}[U\times V]$.
    Then we have
    \begin{multline*}\ind (\mathbb{C}[U\times V])\cong \Big\{\sum_{A\in M(m,n;d)}c_{A}\bullet\xi_{A}\mid c_{A}\in \mathbb{C},\\ c_{A}=\sign(A,\tau,\sigma)c_{(\tau,\sigma)A}\text{ for all }(\tau,\sigma)\in \stab(u_{0},v_{0})\Big\}\subset S(m,n;d)^{*}\end{multline*} as $S(m,n;d)$-supermodules. The $\mathbb{Z}_2$-grading of the supermodule on the right hand side of the isomorphism is induced from the $\mathbb{Z}_2$ grading of $S(m,n;d)^{*}\cong A(m,n;d)$.
\end{lemma}
\begin{proof}
    Let $B=\{\xi_{A}\mid A\in M(m,n;d)\}$. Consider the following isomorphism of superspaces $$\phi:\Hom_{\mathbb{C}}(S(m,n;d),\mathbb{C}[U\times V])\xrightarrow[]{\cong} \Hom_{\mathbb{C}}(U\times V\times B,\mathbb{C})$$
    defined as follows: let
    $$f=\sum_{A\in M(m,n;d)}\big(\sum_{(u,v)\in U\times V}c_{u,v,A}\cdot(u,v)\big)\bullet \xi_{A}\in\Hom_{\mathbb{C}}(S(m,n;d),\mathbb{C}[U\times V]),$$ then
    $$\phi(f)(u,v,\xi_{A}):=c_{u,v,A}$$ for all $u,v\in U\times V$ and $A\in M(m,n;d)$.
    With respect to this identification, the superspace $\Hom_{S_{m}\times S_{n}}(S(m,n;d),\mathbb{C}[U\times V])$ consists of elements of the form $$\sum_{\substack{(u,v)\in U\times V\\A\in M(m,n;d)}} c_{u,v,A}\bullet(u,v,\xi_{A})$$ satisfying
    \begin{equation}\label{eq: identificationCondition}
        c_{\tau.u,\sigma.v,{(\tau,\sigma)A}}=\sign(A,\tau,\sigma)c_{u,v,{A}}.
    \end{equation}
    Define a left action of $S(m,n;d)$ on $\Hom_{\mathbb{C}}(U\times V\times B,\mathbb{C})$ by requiring
    $$\xi\cdot\phi(f)=\phi(\xi\cdot f)$$ for all $\xi\in S(m,n;d)$ and $f\in \Hom_{\mathbb{C}}(S(m,n;d),\mathbb{C}[U\times V])$. 
    
    Fix $(u_{0},v_{0})\in U\times V$. Then \eqref{eq: identificationCondition} implies that
    \begin{align}\label{compatibility}
c_{u,v,{A}}=\sign((\tau_{0}^{-1},\sigma_{0}^{-1})A,\tau_{0},\sigma_{0})c_{u_{0},v_{0},{(\tau_{0}^{-1},\sigma_{0}^{-1})A}}
    \end{align} if $(u,v)=(\tau_{0},\sigma_{0}).(u_{0},v_{0})$. Define the map $\psi:\ind \mathbb{C}[U\times V]\xrightarrow[]{}\Hom_{\mathbb{C}}(S(m,n;d),\mathbb{C})$ by setting $\psi(f)(\xi_{A})=f(u_{0},v_{0},\xi_{A})$. Clearly this map is compatible with $S_{m}\times S_{n}$ action on both sides. Moreover, $\psi$ preserves the action of $S(m,n;d)$. Equation \eqref{compatibility} implies that $\psi$ is injective. Finally, the image of $\psi$ is given by $$\Big\{\sum_{A\in M(m,n;d)}c_{A}\bullet\xi_{A} \;\Big| \; c_{A}\in \mathbb{C}, c_{A}=\sign(A,\tau,\sigma)c_{(\tau,\sigma)A}\text{ for all }(\tau,\sigma)\in \text{Stab}(u_{0},v_{0})\Big\},$$
     which completes the proof of the lemma.
\end{proof}

\begin{definition}[Partition permutation representations]
Given a partition $\lambda$ of a positive integer $n$, let
\begin{displaymath}
    X_\lambda = \{(S_1, S_2, \dotsc, S_{l(\lambda)})\mid S_1\sqcup S_2\sqcup \dotsb \sqcup S_{l(\lambda)} = [n], \; |S_i|=\lambda_i \text{ for } 1\leq i\leq l(\lambda)\},
\end{displaymath}
the set of all ordered set partitions of $[n]$ with block sizes $\lambda_1,\lambda_2,\dotsc,\lambda_{l(\lambda)}$.
The action of $S_n$ on $[n]$ induces an action of $S_n$ on $X_\lambda$. This gives rise to the permutation representation $\mathbb C[X_\lambda]$ of $S_n$ whose Frobenius characteristic is $h_\lambda$ (see e.g.,~\cite[Exercise 5.5.5]{prasad2015representation}). We call $\mathbb{C}[X_{\lambda}]$ the \textit{partition permutation representation} of $S_n$ corresponding to the partition $\lambda$.
\end{definition}
\begin{proposition}\label{permutationmodules}
    If $\rho_{m}$ is a representation of $S_m$ and $\rho_n$ is a representation of $S_n$, then we have
    $$\ch(\ind(\rho_{m}\otimes \rho_{n}))=(\fr(\rho_{m})[H]\fr(\rho_{n})\{E\})_{d}$$
    where $\fr(\rho_m)$ denotes the Frobenius characteristic of $\rho_m$.
\end{proposition}
\begin{proof}
    We first prove the result for partition permutation representations, i.e., we show that
    for $\mu\vdash m$ and $\nu\vdash n$, 
    \begin{equation}\label{charPermutation}
        \ch(\ind(\mathbb{C}[X_{\mu}]\otimes \mathbb{C}[X_{\nu}]))=(h_{\mu}[H]h_{\nu}\{E\})_{d}.
    \end{equation}
Let $l(\mu)=l$ and $l(\nu)=p$. Let $(u_{0},v_{0})\in X_{\mu}\times X_{\nu}$ be the element where 
    $$u_{0}=(\{1,2,\dotsc,\mu_{1}\},\{\mu_{1}+1,\dotsc,\mu_{1}+\mu_{2}\},\dotsc,\{\mu_{1}+\dotsc\mu_{l-1}+1,\dotsc,\mu_{1}+\dotsc\mu_{l-1}+\mu_{l}\})$$ 
    and 
    $$v_{0}=(\{1,2,\dotsc,\nu_{1}\},\{\nu_{1}+1,\dotsc,\nu_{1}+\nu_{2}\},\dotsc,\{\nu_{1}+\dotsc\nu_{p-1}+1,\dotsc,\nu_{1}+\dotsc\nu_{p-1}+\nu_{p}\}).$$
The stabilizer of $(u_{0},v_{0})$ for the action of $S_{m}\times S_{n}$ is $$S_{\mu}\times S_{\nu}= (S_{\mu_{1}}\times \cdots\times S_{\mu_{l}})\times (S_{\nu_{1}}\times \cdots \times S_{\nu_{p}})\subset S_{m}\times S_{n}.$$
Then Lemma~\ref{lemma:perm-induction} implies that 
\begin{multline*}
    \ind(\mathbb{C}[X_{\mu}\times X_{\nu}])=\{\sum_{A\in M(m,n;d)}f_{A}\bullet\xi_{A}\mid f_{A}\in\mathbb{C},   f_{A}=\sign(A,\tau,\sigma )f_{(\tau,\sigma)A}\\ \text{for all }(\tau,\sigma)\in S_{\mu}\times S_{\nu}\}.
\end{multline*}
Note that $S_{\mu}\times S_{\nu}$ permutes the rows inside consecutive blocks of sizes $\mu_{k}\times (m+n)$ and $\nu_{t}\times (m+n)$ of $A\in M(m,n;d)$, where $k\in [l]$ and $t\in[p]$. Proceeding as in the proof of Proposition \ref{trivial}, we obtain
\begin{equation}\label{permutationCharacter}
    \ind(\mathbb{C}[X_{\mu}\times X_{\nu}])=\sum_{\substack{(\alpha,\beta)\in\mathbb{N}^m \times\mathbb N^n\\|\alpha|+|\beta|=d}} |P^{\zeta}_{\theta}(\alpha,\beta;\mu,\nu;d)|x^{\alpha}y^{\beta},
\end{equation}
where $P^{\zeta}_{\theta}(\alpha,\beta;\mu,\nu;d)\subset M^{\zeta}_{\theta}(\alpha,\beta;d)$ consists of matrices $A\in M^{\zeta}_{\theta}(\alpha,\beta;d)$ such that
$$A_{i*}\leq A_{(i+1)*} \text{ } (\text{with respect to the lexicographic ordering})$$ for all $i\in \bigcup_{k=0}^{l-1}\{\mu_{k}+1,\cdots,\mu_{k+1}\} \cup\bigcup_{t=0}^{p-1}\{\nu_{t}+1,\cdots,\nu_{t+1}\}$, where $\mu_{0}=\nu_{0}:=0$, i.e., the rows of $A$ weakly increase inside consecutive blocks of sizes $\mu_{k}\times (m+n)$ ($1\leq k\leq l$) and $\nu_{t}\times (m+n)$ ($1\leq t\leq p$). The right hand side of \eqref{permutationCharacter} is the degree $d$ homogeneous part of
\begin{align*}
    &\sum_{(\alpha,\beta)\in\mathbb{N}^m \times\mathbb N^n} x^{\alpha}y^{\beta}\sum_{\substack{\sum_{k=1}^{l}a_{k}+\sum_{t=1}^{p}c_{t}=\alpha,\\\sum_{k=1}^{l}b_{k}+\sum_{t=1}^{p}d_{t}=\beta}}\prod_{k=1}^{l}|\tilde{P}_{\mu_{k}}(m,n;a_k ,b_k)|\prod_{t=1}^{p}|\tilde{Q}_{\nu_{t}}(m,n;c_t ,d_t)|\\
    =& \bigg(\prod_{k=1}^{l}\sum_{(a_k, b_k)\in\mathbb{N}^m \times\mathbb N^n}|\tilde{P}_{\mu_{k}}(m,n;a_{k},b_{k})|x^{a_k}y^{b_k}\bigg)\bigg(\prod_{t=1}^{p}\sum_{(c_t, d_t)\in\mathbb{N}^m \times\mathbb N^n}|\tilde{Q}_{\nu_{t}}(m,n;c_t ,d_t)|x^{c_t}y^{d_t}\bigg)\\
    =&\prod_{k=1}^{l} h_{\mu_{k}}[H]\prod_{t=1}^{p} h_{\nu_{t}}\{E\} = h_{\mu}[H]h_{\nu}\{E\},
\end{align*}
where the second equality follows from Lemma \ref{lem2}. This proves the proposition when $\rho_m$ and $\rho_n$ are partition permutation modules.

Since the set of homogeneous symmetric functions forms a basis for the algebra of symmetric functions, $\rho_m$ and $\rho_n$ can be uniquely expressed in terms of $\mathbb{C}[X_{\mu}]$, $\mu\vdash m$ and $\mathbb{C}[X_{\nu}]$, $\nu\vdash n$, respectively. Hence, $\rho_m\otimes \rho_n$ can be expanded uniquely in terms of $\mathbb{C}[X_{\pi}]\otimes \mathbb{C}[X_{\eta}]$, where $\pi\vdash m$ and $\eta\vdash n$. Now apply \eqref{charPermutation} to each such $\mathbb{C}[X_{\mu}]\otimes \mathbb{C}[X_{\nu}]$ summand to obtain the result.
\end{proof}

When $\rho_m$ and $\rho_n$ are chosen to be Specht modules in Proposition \ref{permutationmodules}, we have the following corollary.

\begin{corollary}\label{inductionFromula}
For $\mu\vdash m$ and $\nu\vdash n$, we have
$$\ch(\ind(V_{\mu}\otimes V_{\nu}))=(s_{\mu}[H]s_{\nu}\{E\})_{d}.$$
\end{corollary}

Recall from the Introduction that $r_{\lambda\mu\nu}$ denotes the coefficient of $V_{\mu}\times V_{\nu}$ when $W_{\lambda}$ is  restricted to $S_{m}\times S_{n}$. Then Proposition \ref{frobeniusReciprocity} and Proposition \ref{permutationmodules} imply the following super version of Littlewood's formula.
\begin{theorem}[Super-Littlewood identity]\label{Littlewood}
    For $\lambda\in\Hi$, $\mu\vdash m$ and $\nu\vdash n$, we have
    $$r_{\lambda\mu\nu}=\langle s_{\mu}(X)s_{\nu}(Y), \fr(\res(W_{\lambda}))\rangle_{\Lambda(X,Y)} = \langle  s_{\mu}[H]s_{\nu}\{E\}, hs_{\lambda}(X;Y) \rangle_{\Lambda(X|Y)}.$$ 
\end{theorem}

\section{Unimodality of bipartite superpartitions}\label{bipartiteNumbers}
Let $p_{m}(k,l)$ denote the number of vector partitions (also known as bipartite partitions) of length at most $m$ of the vector $(k,l)\in \mathbb{N}^2$. Kim and Hahn \cite{kim1997partitions} proved that if $k\geq l$, then
$$p_{m}(k,l)\geq p_{m}(k+1,l-1)$$
for all $m\in\mathbb{N}$. In \cite{narayanan2021polynomial}, the authors provide an alternative proof of the result using the positivity of the classical restriction coefficients. In this section, following the approach of \cite{narayanan2024hook}, we generalize Kim and Hahn's result to the setting of \textit{bipartite superpartitions} (see Proposition \ref{unimodalitySuperPartitions}).

We abuse notation by using $\eta_{m}$ (resp. $\epsilon_{m}$) to denote both the identity character (resp. sign character) of $S_m$ and the partition $(m)$ (resp. $(1^{m})$). We begin by expressing certain superrestriction coefficients as a signed sum of the cardinalities of sets of matrices defined in Definition \ref{refinedSuperMatrices}.

\begin{proposition}\label{trivialRestriction}
    For any $\lambda\in\Hi$ and $\zeta\in\{\eta_{m},\epsilon_{m}\}$ and $\theta\in\{\eta_{n},\epsilon_{n}\}$,  

    \begin{enumerate}
        \item if $l\leq m$, then 
        $$r_{\lambda\zeta\theta}=\sum_{\tau\in S_{m}}\sign(\tau)\cdot |P^{\zeta}_{\theta}(\lambda+\delta_{m}-\tau\cdot\delta_{m},0;d)|,$$
        \item if $\lambda_{1}\leq n$, then 
        $$r_{\lambda\zeta\theta}=\sum_{\sigma\in S_{n}}\sign(\sigma)\cdot |P^{\zeta}_{\theta}(0,\lambda'+\delta_{n}-\sigma\cdot\delta_{n};d)|,$$
    \end{enumerate}

    where $\delta_{m}=(m-1,m-2,\dotsc,1)$ and $\tau\cdot\delta_{m}=(m-\tau(1),\dotsc,m-\tau(m))$. In the above expression, we adopt the convention that, for any $\alpha\in\mathbb{Z}^{m}$ and $\beta\in\mathbb{Z}^{n}$, $P^{\zeta}_{\theta}(\alpha,\beta;d)=0$ if any of the parts of $\alpha$ or $\beta$ is negative.
\end{proposition}
\begin{proof}
Recall from Proposition \ref{trivial} that $$\ch(\ind \zeta\otimes \theta)=\sum_{\substack{(\alpha,\beta)\in\mathbb{N}^{m}\times\mathbb{N}^{n}\\|\alpha|+|\beta|=d}} |P^{\zeta}_{\theta}(\alpha,\beta;d)|x^{\alpha}y^{\beta}.$$ Then from Theorem \ref{Littlewood}, we have
\begin{align*}
\sum_{\substack{(\alpha,\beta)\in\mathbb{N}^{m}\times\mathbb{N}^{n}\\|\alpha|+|\beta|=d}} |P^{\zeta}_{\theta}(\alpha,\beta;d)|x^{\alpha}y^{\beta}&=\sum_{\pi\in \Hi}r_{\pi\zeta\theta}hs_{\pi}(X,Y)\\
&=\sum_{\pi\in\Hi}r_{\pi\zeta\theta}(\sum_{\mu,\nu}c^{\pi}_{\mu,\nu'}s_{\mu}(X)s_{\nu}(Y))\\
&=\sum_{\mu,\nu}(\sum_{\pi\in\Hi}r_{\pi\zeta\theta}c^{\pi}_{\mu\nu'})s_{\mu}(X)s_{\nu}(Y)\\
\end{align*}
Using the Cauchy bialternant formula for Schur polynomials \cite[Equation 3.1]{macdonald1998symmetric}, when $l(\mu)\leq m$ and $l(\nu)\leq n$, we obtain
\begin{align*}
    a_{\delta_{m}}(X)a_{\delta_{n}}(Y)\sum_{\substack{(\alpha,\beta)\in\mathbb{N}^{m}\times\mathbb{N}^{n}\\|\alpha|+|\beta|=d}} |P^{\zeta}_{\theta}(\alpha,\beta;d)|x^{\alpha}y^{\beta}&=\sum_{\mu,\nu}(\sum_{\pi\in\Hi}r_{\pi\zeta\theta}c^{\pi}_{\mu\nu'})a_{\mu+\delta_{m}}(X)a_{\nu+\delta_{n}}(Y)
\end{align*}
 where $a_{\mu+\delta_{m}}(X)=\det (x^{\mu_{j}+m-j})$ and $a_{\delta_{m}}(X)= \det (x^{m-j})=\sum_{\tau\in S_{m}}\sign(\tau)x^{\tau\cdot\delta_{m}}$. Using the definitions of $a_{\delta_{m}}(X)$ and $a_{\delta_{n}}(Y)$, we have

\begin{align*}
    &\sum_{(\tau,\sigma)\in S_{m}\times S_{n}}\sign (\tau)\sign (\sigma)\sum_{\substack{(\alpha,\beta)\in\mathbb{N}^{m}\times\mathbb{N}^{n}\\|\alpha|+|\beta|=d}} |P^{\zeta}_{\theta}(\alpha,\beta;d)|x^{\alpha+\tau\cdot\delta_{m}}y^{\beta+\sigma\cdot\delta_{n}}\\=&\sum_{\mu,\nu}(\sum_{\pi\in\Hi}r_{\pi\zeta\theta}c^{\pi}_{\mu\nu'})a_{\mu+\delta_{m}}(X)a_{\nu+\delta_{n}}(Y).
\end{align*}

Now, comparing the coefficients of $x^{\mu+\delta_{m}}y^{\nu+\delta_{n}}$ on both sides,
we obtain
\begin{align}\label{eq}
    \sum_{\pi\in\Hi}r_{\pi\zeta\theta}c_{\mu\nu'}^{\pi}=\sum_{(\tau,\sigma)\in S_{m}\times S_{n}}\sign(\tau)\sign(\sigma)\cdot |P^{\zeta}_{\theta}(\mu+\delta_{m}-\tau\cdot\delta_{m},\nu+\delta_{n}-\sigma\cdot\delta_{n};d)|
\end{align}
for every $\mu$ and $\nu$. Denote the right hand side of \eqref{eq} by $\mathcal{P}(\mu,\nu;d)$ for simplicity.

When $l(\lambda)\leq m$, we take $\mu=\lambda$ and $\nu=\emptyset$ in \eqref{eq} to obtain
$$\mathcal{P}(\lambda,\emptyset;d)=\sum_{\pi\in\Hi}c^{\pi}_{\lambda,\emptyset}r_{\pi\zeta\theta}=c^{\lambda}_{\lambda,\emptyset}r_{\lambda\zeta\theta}=r_{\lambda\zeta\theta}$$
since $c^{\lambda}_{\lambda\emptyset}=1$ and $c^{\pi}_{\lambda\emptyset}=0$ unless $\lambda\subset \pi$.

Similarly, when $\tilde\lambda \in\Hi$ satisfies $\tilde\lambda_{1}\leq n$, we take $\mu=\emptyset$ and $\nu={\tilde{\lambda}}^{'}$ in \eqref{eq} to obtain
$$\mathcal{P}(\emptyset,\tilde\lambda^{'} ;d)=\sum_{\pi\in\Hi}c^{\pi}_{\emptyset,\tilde\lambda}r_{\pi\zeta\theta}=c^{\tilde\lambda}_{\emptyset,\tilde\lambda}r_{\tilde\lambda\zeta\theta}=r_{\tilde\lambda\zeta\theta}.$$ This completes the proof.
\end{proof}

\begin{definition}[Bipartite superpartitions]\label{bipartiteSuperPartitions}
    Let $\alpha\in\mathbb{N}^{2}$, $m,n\in\mathbb{N}$.
    \begin{enumerate}
        
    \item Let $B^{\eta_m}_{\eta_n}(\alpha)$ (resp. $B^{\eta_m}_{\epsilon_n}(\alpha)$) denote the set of tuples $(v_{1},\dotsc,v_{m},v_{1}',\dotsc, v_{n}')$ with each $v_{i}\in\mathbb{N}^{2}$ and $v_{j}'\in\mathbb{B}^{2}$ such that
    \begin{enumerate}
        \item $(v_i)_{i=1}^{m}$ and $(v'_j)_{i=1}^{n}$ are weakly increasing in lexicographic order.
        \item any element from the set $\{(1,0),(0,1)\}$ (resp. $\{(1,1), (0,0)\}$) can appear at most once in $\{v_{1}'\dotsc,v_{n}'\}$,
        \item $\sum_{i=1}^{m}v_{i}+\sum_{j=1}^{n}v_{j}'=\alpha$.
    \end{enumerate}
    \item Let $B^{\epsilon_m}_{\eta_n}(\alpha)$ (resp. $B^{\epsilon_m}_{\epsilon_n}(\alpha)$) denote the set of tuples $(v_{1},\dotsc,v_{m},v_{1}',\dotsc, v_{n}')$ with each $v_{i}\in\mathbb{N}^{2}$ and $v_{j}'\in\mathbb{B}^{2}$ such that
    \begin{enumerate}
        \item $(v_i)_{i=1}^{m}$ and $(v'_j)_{i=1}^{n}$ are weakly increasing in lexicographic order.
        \item $v_{i}\neq v_{j}$ for all $1\leq i<j\leq m$.
        \item any element from the set $\{(1,0),(0,1)\}$ (resp. $\{(1,1), (0,0)\}$) can appear at most once in $\{v_{1}'\dotsc,v_{n}'\}$,
        \item $\sum_{i=1}^{m}v_{i}+\sum_{j=1}^{n}v_{j}'=\alpha$.
    \end{enumerate}
    \end{enumerate}
\end{definition}

\begin{proposition}[Unimodality of bipartite superpartitions]\label{unimodalitySuperPartitions}Assume that $m,n\in\mathbb{N}$ and $\lambda=(\lambda_{1},\lambda_{2})$ be any partition with exactly two parts. Then for $\zeta\in\{\eta_m ,\epsilon_m\}$ and $\theta\in\{\eta_n ,\epsilon_n\}$, we have
\begin{align}\label{unimodality}
    |B^{\zeta}_{\theta}(\lambda_{1},\lambda_{2})|\geq |B^{\zeta}_{\theta}(\lambda_{1}+1,\lambda_{2}-1)|
\end{align}   
\end{proposition}
\begin{proof}
Assume that $m\geq 2, n\in\mathbb{N}$. Let $\lambda=(\lambda_{1},\lambda_{2})\vdash d$. Pad $\lambda$ with $m-2$ zero's to obtain $\lambda=(\lambda_{1},\lambda_{2},0^{m-2})$. From Corollary \ref{trivialRestriction}.(1), we have
\begin{multline}\label{restrictioncoeff}r_{\lambda\zeta\theta}=\sum_{\tau\in S_{m}}\sign(\tau)\cdot|P^{\zeta}_{\theta}((\lambda_{1}+\tau(1)-1,\lambda_{2}+\tau(2)-2,\tau(3)-3,\dotsc,\tau(m)-m),0;d)|.
\end{multline}

If $\tau(i)-i\geq 0$ for all $i>2$, then we must have $\tau(i)-i= 0$ for all $i>2$. Hence \eqref{restrictioncoeff} reduces to
\begin{align*}
    r_{\lambda\zeta\theta}=&|P^{\zeta}_{\theta}((\lambda_{1},\lambda_{2}),0;d)|-|P^{\zeta}_{\theta}((\lambda_{1}+1,\lambda_{2}-1),0;d)|\\
    =&|B^{\zeta}_{\theta}(\lambda_{1},\lambda_{2})|-|B^{\zeta}_{\theta}(\lambda_{1}+1,\lambda_{2}-1)|.
\end{align*}
Since $r_{\lambda\zeta\theta}\geq 0$, we have the required identity in this case.

Assume that $m=0$, $n\in\mathbb{N}$ and $\theta=\epsilon_n$. For $v\in\mathbb{N}^{2}$ and $i\in[2]$, let $(v)_{i}$ denote the $i$th coordinate of $v$. We define an injective map from $B^{\zeta}_{\epsilon_{n}}(\lambda_{1}+1,\lambda_{2}-1)$ to $B^{\zeta}_{\epsilon_{n}}(\lambda_{1},\lambda_{2})$ as follows: define $$(v''_{1},\dotsc, v''_{n})\mapsto (v'_{1},\dotsc, v'_{n}),$$
where $(v'_{1},\dotsc, v'_{n})$ is obtained from $(v''_{1},\dotsc, v''_{n})$ by replacing the leftmost $(1,0)$ in\\ $(v''_{1},\dotsc, v''_{n})$ (which exists by hypothesis) with $(0,1)$. To show the well-definedness of the map, we have to show that for any $(v''_{1},\dotsc, v''_{n})\in B^{\zeta}_{\epsilon_{n}}(\lambda_{1}+1,\lambda_{2}-1)$, there exists $i\in[n]$ such that $v''_{i}=(1,0)$. If possible, assume that $v''_i\neq (1,0)$ for all $i\in[n]$. Then the multiset $\{\{v''_{1},\cdots,v''_{n}\}\}$ would only consists of $(0,1)$'s, at most one $(1,1)$ and at most one $(0,0)$. Then we would have $$\lambda_{2}-1= \sum_{i=1}^{n}(v''_{i})_2 \geq \sum_{i=1}^{n}(v''_{i})_1 = \lambda_{1}+1,$$ a contradiction.  This proves \eqref{unimodality} in this case.

If $m=0$, $n\in\mathbb{N}$ and $\theta=\eta_n$, we have $$|B^{\zeta}_{\eta_{n}}(\lambda_{1}+1,\lambda_{2}-1)|=0$$ for $\zeta\in\{\eta_{m},\epsilon_{m}\}$.  Hence \eqref{unimodality} holds in this case.

Assume that $m=1, n\in\mathbb{N}$ and $\theta=\eta_n$. We construct an injective map from $B^{\zeta}_{\eta_{n}}(\lambda_{1}+1,\lambda_{2}-1)$ to $B^{\zeta}_{\eta_{n}}(\lambda_{1},\lambda_{2})$ as follows: define 
$$(v,v''_{1},\dotsc,v''_{n})\mapsto (w,v''_{1},\dotsc,v''_{n}),$$
where $w=v+(-1,1)$. To show the well-definedness of the map, we have to show that $(v)_{1}>0$ if $(v,v''_{1},\dotsc,v''_{n})\in B^{\zeta}_{\eta_{n}}(\lambda_{1}+1,\lambda_{2}-1)$. If possible, assume that $(v)_{1}=0$. Since the multiset $\{\{v''_{1},\cdots,v''_{n}\}\}$ consists of at most one $(1,0)$, we would have $$1\geq \bigg((v)_{1}+\sum_{i=1}^{n}(v''_{i})_1\bigg)-\bigg((v)_{2} + \sum_{i=1}^{n}(v''_{i})_2\bigg) = (\lambda_{1}+1)-(\lambda_{2}-1)=\lambda_{1}-\lambda_{2}+2\geq 2,$$ a contradiction.
Hence \eqref{unimodality} also holds in this case.

Now, assume that $m=1, n\in\mathbb{N}$, $\theta=\epsilon_n$. Let $(v,v'_{1},\dotsc v'_{n})\in B^{\zeta}_{\epsilon_{n}}(\lambda_{1},\lambda_{2})$.
Define $$b_{0}:B^{\zeta}_{\epsilon_{n}}(\lambda_{1},\lambda_{2})\xrightarrow[]{}\mathbb{B}$$ by
\begin{equation*}
b_{0}(v,v'_{1},\dotsc, v'_{n})=
\begin{cases}
    1 & \text{if }v'_{i}=(0,0)\text{ for some }i\in[n]\\
    0 & \text{otherwise.}
\end{cases}
\end{equation*}
We construct a surjective map $$f: B^{\zeta}_{\epsilon_{n}}(\lambda_{1},\lambda_{2})\xrightarrow{} B^{\zeta}_{\epsilon_{n}}(\lambda_{1}+1,\lambda_{2}-1)$$
by
\begin{equation*}
    f(v,v'_{1},\dotsc,v'_{n})=
    \begin{cases}
        (w,v'_{1},\dotsc,v'_{n}) & \text{if }n- b_{0}(v,v'_{1},\dotsc v'_{n})< \lambda_{2}\\
        (w,v'_{1},\dotsc,v'_{n}) & \text{if }\lambda_{2}\leq n- b_{0}(v,v'_{1},\dotsc v'_{n})\leq \lambda_{1}\text{ and }(v)_{2}>0\\
        (v,v''_{1},\dotsc,v''_{n}) & \text{if }\lambda_{2}\leq n- b_{0}(v,v'_{1},\dotsc v'_{n})\leq \lambda_{1}\text{ and }(v)_{2}=0\\
        (v,v''_{1},\dotsc,v''_{n}) & \text{if } n- b_{0}(v,v'_{1},\dotsc v'_{n})> \lambda_{1},
    \end{cases}
\end{equation*}
where $w=v+(1,-1)$ and $(v''_{1},\dotsc v''_{n})$ is obtained from $(v'_{1},\dotsc v'_{n})$ by replacing the rightmost $(0,1)$ with $(1,0)$. 

\textbf{$\phi$ is well-defined:} Let $(v,v'_{1},\dotsc, v'_{n})\in B^{\zeta}_{\epsilon_{n}}(\lambda_{1},\lambda_{2})$. If $n- b_{0}(v,v'_{1},\dotsc, v'_{n})< \lambda_{2}$, we need to show that $(v)_{2}>0$. If $(v)_{2}=0$, then the multiset $\{\{(v'_{1})_{2},\dotsc, (v'_{n})_{2}\}\}$ of second coordinates would contain $\lambda_{2}$ many $1$'s, which would imply that $n- b_{0}(v,v'_{1},\dotsc v'_{n})\geq \lambda_{2}$, contradicting our hypothesis. The well-definedness of $f$ is clear when $\lambda_{1}\leq n- b_{0}(v,v'_{1},\dotsc, v'_{n})\leq \lambda_{2}$.
If $n- b_{0}(v,v'_{1},\dotsc, v'_{n})>\lambda_{1}$ and $v'_{i}\neq (0,1)$ for all $i\in[n]$, then the following multiset $\{\{(v'_{1})_{1},\dotsc,(v'_{n})_{1}\}\}-\{\{0\}\}$ of first coordinates would only consist of $1$'s, which would imply that $n-b_{0}(v,v'_{1},\dotsc, v'_{n})\leq \lambda_{1}$, contradicting our hypothesis.

\textbf{$\phi$ is surjective:} Let $(w,v''_{1},\dotsc, v''_{n})\in B^{\zeta}_{\epsilon_{n}}(\lambda_{1}+1,\lambda_{2}-~1)$. 

\textbf{Case 1:} If $\lambda_{1}\geq n-f(w,v''_{1},\dotsc, v''_{n})$, then we must have $(w)_{1}>0$. If not, then the multiset $\{\{(v''_{1})_{1},\dotsc, (v''_{n})_{1}\}\}$ of first coordinates would contain $\lambda_{1}+1$ many $1$'s, which would imply that 
$$n-f(w,v''_{1},\dotsc, v''_{n})\geq \lambda_{1}+1>\lambda_{1},$$ contradicting our hypothesis. Let $v=w+(-1,1)$. Then $$f(v,v''_{1},\dotsc, v''_{n})=(w,v''_{1},\dotsc, v''_{n}).$$ 

\textbf{Case 2:} If $\lambda_{1}< n-b_{0}(f(w,v''_{1},\dotsc, v''_{n}))$, then there must exist at least one $i\in[n]$ such that $v''_{i}=(1,0)$. If not, then the multiset $\{\{(v''_{1})_{2},\dotsc, (v''_{n})_{2}\}\}-\{\{0\}\}$ of second coordinates would only consist of $1$'s, which would imply that
$$n-b_{0}(f(w,v''_{1},\dotsc, v''_{n}))\leq \lambda_{2}-1<\lambda_{1}.$$ Let $(v'_{1},\dotsc v'_{n})$ be the vector obtained from $(v''_{1},\dotsc, v''_{n})$ by replacing the leftmost $(1,0)$ with $(0,1)$. Then $f(w, v'_{1},\dotsc, v'_{n})=(w,v''_{1},\dotsc, v''_{n})$. Hence $f$ is surjective.

This finishes the proof of the proposition.
\end{proof}

\section{Relation to classical restriction coefficients}
Let $\{\tilde{W}^{m}_{\lambda}\mid l(\lambda)\leq m \}$ denote the set of simple modules for the Schur algebra $S(m;d)$.
The $0$-graded part of $S(m,n;d)$ is isomorphic to the product $S(m;d)\times S(n;d)$ of classical Schur algebras. 
The restriction of $W_{\lambda}$ to $S(m;d)\times S(n;d)$ is given by \cite[Theorem~6.11]{berele1987hook}:
\begin{align}\label{eq1}
    \text{Res}^{S(m,n;d)}_{S(m;d)\times S(n;d)}W_{\lambda}=\bigoplus_{\pi,\eta}(\tilde{W}^{m}_{\pi}\otimes \tilde{W}^{n}_{\eta})^{\oplus c^{\lambda}_{\pi,\eta'}}\hbox{ for }\lambda\in\mathcal{H}(m,n;d).
\end{align}

The classical restriction coefficients $r_{\lambda\mu}\in\mathbb{N}$ are defined using the identity
\begin{align*}
    \text{Res}^{S(m;d)}_{S_m}\tilde{W}^{m}_{\lambda}=\bigoplus_{\mu\vdash m}V_{\mu}^{\oplus r_{\lambda\mu}}.
\end{align*} 
For $n=0$, Corollary \ref{Littlewood} reduces to the classical Littlewood idenity $$r_{\lambda\mu}=\langle s_{\lambda}, s_{\mu}[\tilde{H}]_{d}\rangle_{\Lambda(X)}$$ where $\tilde{H}:=\tilde{H}(1)=\sum_{i\geq 0}h_{i}=\prod_{i=1}^{n}(1-x_{i})^{-1}$. Then we have 
\begin{align}\label{eq2}
    \text{Res}^{S(m;d)\times S(n;d)}_{S_{m}\times S_{n}} \Tilde{W}^{m}_{\pi}\otimes \Tilde{W}^{n}_{\eta}=\bigoplus_{\substack{\mu\vdash m,\\ \nu\vdash n}} (V_{\mu}\otimes V_{\nu})^{\oplus r_{\pi\mu}r_{\eta\nu}}. 
\end{align}
From \eqref{eq1} and \eqref{eq2}, we have
\begin{align*}
    \res W_{\lambda}=\bigoplus_{\substack{\mu\vdash m\\\text{ }\nu\vdash n}}\bigoplus_{\pi,\eta} (V_{\mu}\otimes V_{\nu})^{\oplus c^{\lambda}_{\pi\eta'}r_{\pi\mu}r_{\eta\nu}}
\end{align*}
for $\lambda\in\mathcal{H}(m,n;d)$. Using Theorem \ref{Littlewood}, we have the following equality.
\begin{proposition}\label{plethysticIdentity}For $\lambda\in\Hi,\mu\vdash m$ and $\nu\vdash n$, we have
    $$r_{\lambda\mu\nu}=\langle hs_{\lambda},s_{\mu}[H]s_{\nu}\{E\}\rangle_{\Lambda(X|Y)}
    =\sum_{\pi,\eta}c^{\lambda}_{\pi\eta'}{\langle s_{\pi}, s_{\mu}[\tilde{H}]\rangle_{\Lambda(X)}}{\langle s_{\eta}, s_{\nu}[\tilde{H}]}\rangle_{\Lambda(X)},$$ where $\langle\cdot,\cdot\rangle_{\Lambda(X)}$ is the Hall inner product on $\Lambda(X)$.
\end{proposition}
For a partition $\nu=(\nu_{1},\nu_{2},\dotsc,\nu_{l})$ and $n\geq |\nu|+\nu_{1}$, define
$$\nu[n]=(n-|\nu|,\nu_{1},\nu_{2},\dotsc,\nu_{l}).$$ Let $\mu$ and $\nu$ be any two partitions, and let $\lambda\vdash d$. Define $r_{\lambda\mu\nu}(m,n)=r_{\lambda\mu[m]\nu[n]}$ for large enough $m$ and $n$ (requiring $\lambda\in\Hi$). It is well-known that the sequence $(r_{\lambda\nu[n]})_{n\geq 0}$ eventually stabilizes for large enough $n$ (see e.g., \cite{littlewood1977theory,assaf2020specht,narayanan2021character}). This result along with Proposition \ref{plethysticIdentity} implies an analogous stability property for the superrestriction coefficients.
\begin{corollary}{\em(Stability of superrestriction coefficients)} Let $\mu$ and $\nu$ be any two partitions, and let $\lambda\vdash d$. Then the sequence $(r_{\lambda\mu\nu}(m,n))_{m,n}$ eventually stabilizes for large enough $m$ and $n$.
\end{corollary}

\bibliographystyle{alpha}
\bibliography{ref}
\end{document}